\documentclass{amsart}
\usepackage{amssymb, latexsym, euscript}
\usepackage[usenames,dvipsnames,svgnames,table]{xcolor}
\usepackage{lineno}

\def\sA{{\mathfrak A}}   \def\sB{{\mathfrak B}}   
\def\sD{{\mathfrak D}}      
   \def\sH{{\mathfrak H}}   
   \def\sK{{\mathfrak K}}   \def\sL{{\mathfrak L}}
   \def\sN{{\mathfrak N}}

      \def\dC{{\mathbb C}}

   \def\dN{{\mathbb N}}

\def\bB{{\mathbf B}}
\def\bL{{\mathbf L}}

\def\bm\chi{\mbox{\boldmath$\chi$}}
\def\half{{\frac{1}{2}}}

\def\col{{\rm col\,}}

\def\ker{{\rm ker\,}}
\def\cker{{\rm \overline{ker}\,}}
\def\ran{{\rm ran\,}}
\def\cran{{\rm \overline{ran}\,}}
\def\dom{{\rm dom\,}}
\def\mul{{\rm mul\,}}

\def\cdom{{\rm \overline{dom}\,}}

\let\xker=\ker \def\ker{{\xker\,}}
\def\span{{\rm span\,}}

\def\uphar{{\upharpoonright\,}}

\DeclareMathOperator{\hplus}{\, \widehat + \,}
\DeclareMathOperator{\hoplus}{\, \widehat \oplus \,}

\newtheorem{theorem}{Theorem}[section]
\newtheorem{proposition}[theorem]{Proposition}
\newtheorem{corollary}[theorem]{Corollary}
\newtheorem{lemma}[theorem]{Lemma}
\newtheorem{definition}[theorem]{Definition}

\theoremstyle{definition}
\newtheorem{example}[theorem]{Example}
\newtheorem{remark}[theorem]{Remark}

\numberwithin{equation}{section}

\begin{document}

\title[Monotone sequences of operators]
{Sequences of operators, monotone in the sense of contractive domination}

\author[S.~Hassi]{S.~Hassi}
\author[H.S.V.~de~Snoo]{H.S.V.~de~Snoo}

\address{Department of Mathematics and Statistics \\
University of Vaasa \\
P.O. Box 700, 65101 Vaasa \\
Finland} \email{sha@uwasa.fi}

\address{Bernoulli Institute for Mathematics, Computer Science and Artificial Intelligence \\
University of Groningen \\
P.O. Box 407, 9700 AK Groningen \\
Nederland}
\email{h.s.v.de.snoo@rug.nl}

\dedicatory{Dedicated to the memory of V.E. Katsnelson}

\date{}
\thanks{The second author is grateful to the University of Vaasa for its hospitality
when a final version of the present paper was being prepared.}




\begin{abstract}
A sequence of operators $T_n$ from a Hilbert space $\sH$
to Hilbert spaces $\sK_n$
which is nondecreasing in the sense of contractive domination
is shown to have a limit which is still a linear operator $T$ from $\sH$
to a Hilbert space $\sK$.
Moreover, the closability or closedness of $T_n$ is preserved in the limit.
The closures converge likewise and the connection between the limits is investigated.
There is no similar way of dealing directly with linear relations.
However, the sequence of closures is still nondecreasing
and then the convergence is governed by
the monotonicity principle. There are some related results
for nonincreasing sequences.
 \end{abstract}

\maketitle

\section{Introduction}

Let $T_n \in \bL(\sH, \sK_n)$, $n \in \dN$,  be a sequence of linear operators
from a Hilbert space $\sH$ to a Hilbert space $\sK_n$, which satisfy
\begin{equation}\label{een}
 \dom T_{n+1} \subset \dom T_n \quad \mbox{and} \quad
 \|T_n f\| \leq \|T_{n+1}f\|, \,\, f \in \dom T_{n+1}.
\end{equation}
Here and elsewhere the notation $\bL(\sH, \sK)$ indicates the class
of all linear relations between the Hilbert spaces $\sH$ and $\sK$.
It will be shown that there exists a limit of this sequence, namely a linear operator
$T \in \bL(\sH, \sK)$, whose domain is given by
\[
\dom T= \left\{ \varphi \in
  \bigcap_{n \in \dN} \dom T_n :\, \sup_{n \in \dN} \|T_n \varphi\| < \infty \right\},
\]
while, furthermore,
\[
 \|T_n f\| \nearrow \|Tf\|  \quad \mbox{for all} \quad f \in \dom T.
\]
The limit is uniquely determined up to partial isometries.
Moreover, it will be shown that closability and closedness of the operators $T_n$
are preserved in the limit.
The main idea about the existence of the limit is the notion of a representing
map that was described by Szyma\'nski \cite{Szy87}.
In the present paper the emphasis is on how to construct the limit
of the sequence of operators and to discuss analogous sequences of linear relations.
There is a close connection with similar convergence results
in the context of nonnegative forms by Simon \cite{S3} (see also \cite{RS1}),
but the details will be left for a treatment in \cite{HS2023b} in terms of Lebesgue
decompositions and Lebesgue type decompositions of semibounded forms.

\medskip

The monotonicity in \eqref{een} can also be discussed for the case of linear
relations $T_n \in \bL(\sH, \sK_n)$ by requiring that $T_{n+1}$
contractively dominates $T_n$, i.e., there are  contractions
$C_n \in \bL(\sH_{n+1}, \sH_n)$  which satisfy
\begin{equation}\label{twee}
 C_n T_{n+1} \subset T_n.
\end{equation}
Likewise, this kind of monotonicity is preserved under closures $T_n^{**}$
and under taking regular parts $T_{n, {\rm reg}}$
of the relations $T_n$ (see below). In general there is no convergence result as for operators.
However, the  regular parts $T_{n, {\rm reg}}$ form a nondecreasing sequence of closable operators
(as in \eqref{een})  and one may apply the above mentioned results for operators.
Thanks to the condition \eqref{twee} the sequence of nonnegative selfadjoint relations
$T_n^*T_n^{**}$ is nondecreasing in the usual sense and  the monotonicity principle
may be applied. This connects the various forms of convergence.

\medskip

As mentioned above, in the present paper regular parts  of operators or relations
play an important role.
The regular part $T_{\rm reg}$ of a linear relation
$T \in \bL(\sH, \sK)$ shows up in its Lebesgue decomposition, as follows
\[
T=T_{\rm reg}+T_{\rm sing} \quad \mbox{with} \quad
T_{\rm reg}=(I-P)T, \,\,T_{\rm sing}=PT,
\]
where $P$ stands for the orthogonal projection from $\sK$ onto $\mul T^{**}$;
see \cite{HSS2018}, \cite{HSSS2007}.
Hence $T_{\rm reg}$ is a closable operator, while $T_{\rm sing}$ is singular in the sense that
its closure in the graph sense is the product of closed linear subspaces;
note in particular that  $\ran T_{\rm reg} \perp \mul T^{**}$.
The regular part $T_{\rm reg}$ is the largest closable operator that is  dominated  by $T$
in the sense of contractive domination.
There is an interplay with the closure $T^{**}$ of $T$, given by the formula
\begin{equation}\label{vier}
(T^{**})_{\rm reg}=(T_{\rm reg})^{**},
\end{equation}
see Section \ref{app}.
If the relation $T$ is closed, then $\mul T^{**}=\mul T$ and $T_{\rm reg}$ is the usual
closed orthogonal operator part of $T$, often denoted by $T_{\rm op}$. In this
case, clearly,  $T_{\rm reg} \subset T$ and $T$ has the decomposition
\[
 T=T_{\rm reg} \hplus (\{0\} \times \mul T),
\]
where the sum is componentwise.  Note that
the left-hand side of the identity \eqref{vier} stands
for the orthogonal operator part of $T^{**}$.
In the general case the following identity
\[
T^*T^{**}= (T_{\rm reg})^*(T_{\rm reg})^{**}
\]
expresses the nonnegative selfadjoint relation on the left-hand side in terms
of a similar product of closable operators.

\medskip

The case of a sequence of nonincreasing linear operators will also be discussed
with the same methods. Now closability is not preserved so that the main result
is about a nonincreasing sequence of closed linear operators.

\medskip

The paper is organized as follows.  In Section \ref{reldom} there is brief
review of the notion of contractive domination for relations and operators.
For the convenience of the reader the relevant facts for the monotonicity principle
are reviewed in Section \ref{monop}. The representing map
is discussed in Section \ref{repmap} in an appropriate context.
The convergence results are treated next.
The general case of sequences of linear operators can be found in Section \ref{nondecop},
the special case of sequences of closable operators is treated in Section \ref{nondecclos},
and the general case of sequences of linear relations is given in Section \ref{nondecrel}.
In this last section one can also find the connection with the monotonicity principle.
In Section \ref{Ex} a simple example shows the different behaviours of the various sequences
that have been considered. The approximation of closed linear operators
is considered in Section \ref{closed}.  A brief discussion about nonincreasing sequences
of linear operators or relations can be found in Section \ref{noninc}.
Finally, in Section \ref{app} there is a collection of facts concerning the regular part
of the relations $T^*T$ and $T^*T^{**}$ which are used throughout this paper.

\medskip

In the present paper the interest is in monotone sequences of linear operators
or relations in a Hilbert space.
The above mentioned results have a close connection to work on sequences of operators
in the literature; see \cite{Kato}, \cite[Supplement to VIII.7]{RS1}, and \cite{S3}.
The present work also connects sequences which are monotone in the sense of
contractive domination with the monotonicity principle in its version for semibounded selfadjoint relations
\cite{BHS}. Related results in the context of sequences  of semibounded quadratic forms
will be discussed in \cite{HS2023b} (including the connections to \cite{S3} and \cite{AtE1}).

\section{Contractive domination for linear relations}\label{reldom}

The notion of domination for linear relations was introduced in \cite{HS2015}.
The definition and some basic properties are given here.

\begin{definition}\label{domi}
Let $\sH_A$, $\sH_B$, and $\sH$ be Hilbert spaces,
let $A \in \bL(\sH, \sH_A)$ and  let
$B \in \bL(\sH, \sH_B)$. Then $B$ is said to contractively dominate $A$,
denoted by $A \prec_c B$,
if there exists a contraction $C \in \mathbf{B}(\sH_B, \sH_A)$
such that
\begin{equation}\label{domm}
 CB \subset A.
\end{equation}
\end{definition}

It follows from $C \in \mathbf{B}(\sH_B, \sH_A)$
that $CB=\{ \{f,Cf'\} :\, \{f,f'\} \in B\}$.
Therefore,  \eqref{domm} implies
\begin{equation}\label{dome}
\begin{cases}{}
  \,\,\dom B \subset \dom A,  &
  \ker B \subset \ker A, \\
  \,\, C(\ran B) \subset \ran A, &
   C(\mul B) \subset \mul A.
\end{cases}
\end{equation}
Observe that Definition \ref{domi} implies that
the contraction $C \in \mathbf{B}(\sH_B, \sH_A)$ is only fixed
as a mapping form $\ran B$ to $\ran A$. In fact,
the boundedness of $C$ implies that $C$ takes
$\cran B$ into $\cran A$.  Hence, it may and will be assumed
that
\[
C ((\ran B)^\perp) =\{0\}.
\]

\medskip

Note that if $A$ and $B$ are linear relations
which satisfy $B \subset A$, then
$B$ contractively dominates $A$ with $C=I_{\ran B}$.
In particular, $A$ contractively dominates $A^{**}$.
Finally, the notion of contractive domination is transitive:
 \[
\begin{array}{l}
  A \prec_c B\quad\textrm{and}\quad B \prec_c C\quad \Rightarrow \quad A \prec_c C.
\end{array}
\]
If $A \prec_c B$ with a contraction $C \in \mathbf{B}(\sH_B, \sH_A)$, then
it follows from \eqref{domm} and \cite[Proposition 1.3.9]{BHS} that
\begin{equation}\label{dom1}
 A^{*} \subset B^{*} C^{*} \quad \mbox{and} \quad CB^{**} \subset A^{**}.
\end{equation}
In other words, the second inclusion in \eqref{dom1}
shows that the contractive domination in \eqref{domm} is preserved
with the same operator $C$. In particular, if $A \prec_c B$,
then the following inclusions are valid:  $\ran A^*\subset \ran B^*$
and  $\dom B^{**}\subset \dom A^{**}$.
Recall that in the particular case when $A$ and $B$ in Definition \ref{domi}
are linear operators it is possible to give an equivalent
characterization of contractive domination:  $A \prec_c B$
if and only if
\[
 \dom B \subset \dom A \quad \text{and}
  \quad \|A f\|\le \|B f\|,
 \quad f\in \dom B.
\]

The following result shows that contractive domination is preserved by the regular parts.
This observation goes back to \cite{S3} for the case of nonnegative forms and to  \cite{HSS2018}.
Furthermore, it is shown that there is a converse statement in the case of closed linear relations.

\begin{lemma}\label{regg}
Let $A \in \bL(\sH, \sH_A)$ and $B \in \bL(\sH, \sH_B)$ be linear relations.
Then
\[
A \prec_c B \quad \Rightarrow \quad  A_{\rm reg} \prec_c B_{\rm reg}.
\]
Moreover, if the  linear relations $A$ and $B$ are closed, then
\[
A \prec_c B \quad \Leftrightarrow \quad  A_{\rm reg} \prec_c B_{\rm reg}.
\]
\end{lemma}

\begin{proof}
Assume that $CB \subset A$ with a contraction $C \in \mathbf{B}(\sH_B, \sH_A)$.
By \eqref{dome} the operator $C$ maps $\mul B^{**}$ into $\mul A^{**}$.
Let $P_B$ be the orthogonal projection onto $\mul B^{**}$ and let $P_A$
be the orthogonal projection onto $\mul A^{**}$. Let $\{f,f'\} \in B$
and write $\{f,f'\}=\{f,(I-P_B)f'+P_Bf' \}$ (i.e., the Lebesgue decomposition of $B$).
Here $P_B f' \in \mul B^{**}$ and
one concludes that
\[
\{f, Cf'\} = \{f, C(I - P_B )f' + CP_B f'\} \in A,
\]
where $CP_B f' \in \mul A^{**}$. Now observe that
\[
\{f, (I - P_B )f'\} \in B_{\rm reg} \quad \mbox{and} \quad  \{f, (I - P_A)C(I - P_B )f'\} \in  A_{\rm reg}.
\]
Equivalently,
this leads to  $ [(I - P_A)C]B_{\rm reg} \subset A_{\rm reg}$,
and since ($I - P_A)C$ is a contraction
this implies $A_{\rm reg} \prec_c B_{\rm reg}$.

Let $A \in \bL(\sH, \sH_A)$ and $B \in \bL(\sH, \sH_B)$ be closed linear relations.
Then $A_{\rm reg}$ and  $B_{\rm reg}$, belonging to  $\bB(\sH_A, \sH_B)$,
are the closed linear operator parts.
Assume the inequality $A_{\rm reg}\prec B_{\rm reg}$. Then
there exists a contraction $C \in \bB(\sH_B, \sH_A)$ such that $C B_{\rm reg} \subset A_{\rm reg}$.
Without loss of generality one may take $C \uphar (\ran B_{\rm reg})^\perp=0$.
Then, in particular, $C \uphar \ran P_B=\{0\}$ and it follows from
the Lebesgue decomposition $B=B_{\rm reg}+B_{\rm sing}$ that
\[
CB=CB_{\rm reg} \subset A_{\rm reg}.
\]
Since $A$ is closed, one sees that $A_{\rm reg}\subset A$.
Therefore, $CB \subset A$ and $A \prec_c B$.
\end{proof}

The equivalence in the above theorem is restricted to closed linear relations.
By modifying the notion of domination the condition that the relations are closed
can be relaxed by introducing a weaker form of the Lebesgue decomposition; cf.
\cite{HSS2018}, \cite{HSSz2009}.

\medskip

Contractive domination of closed linear relations can be characterized
in terms of the corresponding nonnegative selfadjoint relations;
see \cite[Theorem~4.4]{HS2015}.
Recall from \cite[Definition 5.2.8]{BHS} that two nonnegative relations
$H_1$ and $H_2$ in $\bL(\sH)$ satisfy $H_1 \leq H_2$ when
\begin{equation}\label{root}
\dom H_2^\half \subset \dom H_1^\half \quad \mbox{and} \quad
\|(H_{1, {\rm reg}})^\half f\| \leq \|(H_{2, {\rm reg}})^\half\|, \,\, f \in \dom H_2^\half.
\end{equation}
With this definition the following theorem is clear.

\begin{theorem}\label{nieuw0}
 Let $A \in \bL(\sH, \sH_A)$ and $B \in \bL(\sH, \sH_B)$ be closed linear relations.
Then the following statements are equivalent
\begin{enumerate}\def\labelenumi{\rm(\roman{enumi})}
\item $A^*A \leq B^*B$;
\item $A \prec_c B$ or, equivalently, $A_{\rm reg} \prec_c B_{\rm reg}$.
\end{enumerate}
 \end{theorem}

\begin{proof}
Let $H_1=A^*A$ and $H_2=B^*B$.
By Lemma \ref{regs+} it follows that there exist partial isometries
$U_1 \in \bL(\sH_A, \sH)$ and $U_2 \in \bL(\sH_B, \sH)$, such that.
\[
\left\{
\begin{array}{l}
 \dom H_1^\half =\dom A, \quad (H_{1, {\rm reg}})^\half=U_1 A_{\rm reg}, \\
 \dom H_2^\half =\dom B, \quad (H_{2, {\rm reg}})^\half=U_2 B_{\rm reg}.
\end{array}
\right.
\]
Therefore by means of \eqref{root} this shows that $A^*A \leq B^*B$, i.e.,  $H_1 \leq H_2$,
is equivalent to the assertions
\[
\left\{
\begin{array}{l}
 \dom B \subset \dom A, \\
 \|A_{\rm reg} h\| \leq  \|B_{\rm reg} h\|, \quad h \in \dom B.
 \end{array}
 \right.
\]
In other words, the inequality $A^*A \leq B^*B$ in (i)
is equivalent to the inequality $A_{\rm reg} \prec_c B_{\rm reg}$ in (ii).
\end{proof}

This characterization makes it possible to apply the monotonicity principle in the next section.

\section{The monotonicity principle}\label{monop}

A linear relation $H \in \bL(\sH)$ is called the \textit{strong graph limit}
of a sequence of linear relations $H_n \in \bL(\sH)$, $n \in \dN$,
if for each $\{h,h'\} \in H$ there exists a sequence $\{h_n, h_n'\} \in H_n$
such that $\{h_n,h_n'\} \to \{h,h'\}$; see \cite[Definition 1.9.1]{BHS}.
The strong graph limit is automatically closed, see \cite[p. 80]{BHS}.
Clearly, if all $H_n$ are symmetric, then $H$ is symmetric.
In particular, if all $H_n$ are nonnegative, then $H$ is nonnegative.

\begin{lemma}
Let $H_n \in \bL(\sH)$ be a sequence of nonnegative selfadjoint relations
and let its strong graph limit $H_\infty$ be nonnegative and selfadjoint.
Then for every $f \in \dom H_{\infty}$  there exists a
sequence $f_n  \in \dom H_{n}$ such that
\[
 f_n \to f \quad \mbox{and} \quad
 \| (H_{n, {\rm reg}})^\half f_n\| \to \| (H_{\infty, {\rm reg}})^\half f\|.
\]
\end{lemma}

\begin{proof}
Let $A \in \bL(\sH)$ be any nonnegative selfadjoint relation with square root $A^\half$.
Recall that $\mul A^\half=\mul A$, so that $(A^\half)_{\rm reg}= (A_{\rm reg})^\half$.
If $\{f,f'\} \in A$, then there exists an element $h \in \sH$ such that $\{f,h\} \in A^\half$
and $\{h, f'\} \in A^\half$, which gives
\begin{equation}\label{sq1}
 (f',f)=\|h\|^2.
\end{equation}
Since $h \in \dom A^\half \subset (\mul A)^\perp$, one sees that $h= (A_{\rm reg})^\half f$.
Therefore, it is clear that \eqref{sq1} may be written as
\begin{equation}\label{sq2}
 (f',f)=(A_{\rm reg} f,f)=\| (A_{\rm reg})^\half f\|^2.
\end{equation}

Now let $f \in \dom H_\infty$, then  $\{f,f'\} \in H_\infty$ for some $f' \in \sH$.
By the strong graph convergence
there exists a sequence
$\{f_n, f_n'\} \in H_n$ such that $f_n \to f$ and $f_n' \to f'$.
Therefore, by definition,
there exist elements $h_n \in \sH$  such that
\[
 \{f_n, h_n\} \in (H_n)^\half \quad \mbox{and} \quad \{h_n, f_n'\} \in (H_n)^\half,
\]
and, likewise, there exists an element $h \in \sH$ such that
\[
\{f,h\} \in (H_\infty)^\half \quad \mbox{and} \quad \{h,f'\} \in (H_\infty)^\half.
\]
Then clearly
\[
 \|h_n\|^2=(f_n',f_n) \to (f',f)=\|h\|^2,
\]
or, equivalently, using \eqref{sq2},
\[
 \|(H_{n, {\rm reg}})^\half f_n\| \to \|(H_{\infty, {\rm reg}})^\half f\|.    \qedhere
\]
\end{proof}

In the case of a nondecreasing sequence of nonnegative selfadjoint
relations $H_n$ there is a much stronger result. First  observe that
\[
 H_m \leq H_n \quad \Leftrightarrow \quad (H_m)^\half \prec_c (H_n)^\half,
\]
due to Theorem \ref{nieuw0}, so that if $H_n$ is nondecreasing, one also has
\[
 (H_{m, {\rm reg}})^\half \prec_c  (H_{n, {\rm reg}})^\half.
\]
The following monotonicity principle will be recalled from \cite[Theorem 3.5]{BHSW2010},
\cite[Theorem 5.2.11]{BHS}.

\begin{theorem}\label{mopr}
Let $H_n \in \bL(\sH)$ be a sequence of nonnegative selfadjoint relations
and assume they satisfy
\[
H_m \leq H_n,  \quad m \leq n.
\]
 Then there exists a nonnegative selfadjoint relation $H_\infty \in \bL(\sH)$  with
\[
H_n \leq H_\infty,  \quad n \in \dN.
\]
 In fact, $H_n \to H_\infty$ in the strong resolvent sense
or, equivalently, in the strong graph sense.
Moreover, the square root of $H_\infty$ satisfies
\begin{equation}\label{vasf}
\dom (H_\infty)^\half=\left\{ \varphi \in \bigcap_{n \in \dN} \dom (H_n)^\half  :\,
  \sup_{n \in \dN} \| (H_{n, {\rm reg}})^\half \varphi\| < \infty \right\}
\end{equation}
and, furthermore,
\begin{equation}\label{vasg}
 \| (H_{n,{\rm reg}})^\half \varphi\| \nearrow \| (H_{\infty, {\rm reg}})^\half \varphi\|,
 \quad \varphi \in \dom (H_\infty)^\half.
\end{equation}
\end{theorem}

Note that the multivalued parts of the relations $H_n$ in Theorem \ref{mopr}
form a nondecreasing sequence. Of course, if all relations $H_n$ in Theorem \ref{incr}
are operators, then the limit $H_\infty$ may still be a linear relation with a nontrivial
multivalued part; see the example below.

\begin{example}\label{haha}
Let $A \in \bL(\sH)$ be a nonnegative selfadjoint operator or relation.
Then it is clear that the sequence $H_n=nA$ of nonnegative selfadjoint operators or relations
is nondecreasing. Hence there exists a nonnegative selfadjoint relation $H_\infty$
such that $H_n \to H_\infty$ is the strong graph sense.
To determine $H_\infty$, let $\{f,g\} \in H_\infty$,
then there exists a sequence $\{f_n, g_n\} \in H_n$
such that $f_n \to f$ and $g_n \to g$.
Here $g_n=n h_n$ with $\{f_n, h_n\} \in A$ and,
clearly, $h_n \to 0$. Since $A$ is closed, this implies $\{f,0\} \in A$.
Furthermore, note that $h_n \in \ran A \subset (\ker A)^\perp$.
Hence $g_n \in (\ker A)^\perp$ which implies $g \in (\ker A)^\perp$.
Therefore, it follows that
\[
 H_\infty=\ker A \times (\ker A)^\perp,
\]
since both relations are selfadjoint. Furthermore one has  $\dom (H_\infty)^\half=\ker A$
and $(H_\infty)_{\rm op}=\ker A \times \{0\}$ (as in \eqref{vasf} and \eqref{vasg}).
\end{example}

For sequences of closed relations which are nondecreasing in the
sense of domination there are close connections with Theorem \ref{mopr}
via Theorem \ref{nieuw0}.

\section{Semi-inner products and representing maps}\label{repmap}

Let $\sH$ be a Hilbert space with inner product $(\cdot, \cdot)$ and
let $\sD \subset \sH$ be a linear subspace which is provided with
a semi-inner product $(\cdot, \cdot)_+$. In the following lemma
it will be shown that such a subspace is generated by
a so-called representing map.  The assertion is inspired
by \cite{Szy87}.

\begin{lemma}\label{aux}
Let $\sH$ be a Hilbert space with inner product $(\cdot, \cdot)$.
Let $\sD \subset \sH$ be a linear subspace which is provided with a semi-inner product
$(\cdot, \cdot)_+$. Then there exists  a representing map
$T \in \bL(\sH, \sK)$, where $\sK$ is a Hilbert space, such that
\[
 (\varphi, \psi)_+=(T \varphi, T \psi)_\sK, \quad \varphi, \psi \in \sD=\dom T.
\]
If $T' \in \bL(\sH, \sK')$, where $\sK'$ is a Hilbert space,
is another representing map  with $\dom T'=\sD$,
then there exists a partial isometry $V \in \bB(\sK, \sK')$
with initial space $\cran T$ and final space $\cran T'$,
such that $T'=VT$.
\end{lemma}

\begin{proof}
Let $\sN$ be the set of neutral elements in $\sD$:
\[
 \sN=\big\{ \varphi \in \sD :\, (\varphi, \varphi)_+=0\,\big\}.
\]
Due to the Cauchy-Schwarz inequality the space $\sN$ is linear.
Hence, one may introduce an inner product on the quotient space
$\sD/\sN$ by
\[
 [\varphi+\sN, \psi+\sN]=(\varphi, \psi)_+, \quad \varphi, \psi \in \sD.
\]
The completion of this quotient space is indicated by $\sK$,
so that $\sK$ is a Hilbert space.
Denote the inner product on $\sK$ by $(\cdot, \cdot)_\sK$,
so that $(\varphi+\sN, \psi+\sN)_\sK= [\varphi+\sN, \psi+\sN]$
for $\varphi, \psi \in \sD$.
Next define the operator $T$ from $\sD \subset \sH$ to $\sK$ by
\[
  T \varphi= \varphi+\sN, \quad \varphi \in \sD.
\]
Then it follows that
\[
 (T\varphi, T \psi)_\sK=[\varphi +\sN, \psi+\sN]
 =(\varphi, \psi)_+, \quad \varphi, \psi \in \sD,
\]
which is the assertion of the lemma.

If $T' \in \bL(\sH, \sK')$, where $\sK'$ is a Hilbert space,
is another representing map  with $\dom T'=\sD$,
then
\[
 (T' \varphi, T'\psi)=(\varphi, \psi)_+,
 \quad \varphi, \psi \in \sD=\dom T'.
\]
Then the linear relation $V$ from $\sK$ to $\sK'$, defined by
\[
 \big\{ \{T \varphi, T' \varphi\} :\, \varphi \in \sD \big\},
\]
is an isometric operator from $\ran T$ onto $\ran T'$,
which can be extended as an isometric operator from $\cran T$ onto $\cran T'$,
such that $T'f=VTf$ holds for all $f\in\sD$.
 To get the desired partial isometry $V$ it remains to
continue the isometric map to $(\ran T)^\perp$ as a zero mapping.
This gives the desired result.
\end{proof}

Let $\sD \subset \sH$ be a linear subspace as in Lemma \ref{aux}.
A sequence $\varphi_n \in \sD$ is said to converge to $\varphi \in \sH$
in the sense of $\sD$, in notation $\varphi_n \to_\sD \varphi$, if
\[
 \varphi_n \to \varphi \quad \mbox{in} \quad \sH \quad \mbox{and} \quad \|\varphi_n-\varphi_m\|_+ \to 0.
\]
Then $\sD$ is called \textit{closable} if for any sequence $\varphi_n \in \sD$ one has
\[
 \varphi_n \to_\sD 0 \quad \Rightarrow \quad \|\varphi_n\|_+ \to 0,
\]
and, likewise,  $\sD$ is called \textit{closed} if for any sequence $\varphi_n \in \sD$ one has
\[
 \varphi_n \to_\sD \varphi \quad \Rightarrow \quad  \varphi \in \sD
 \quad \mbox{and} \quad \|\varphi_n -\varphi\|_+ \to 0.
\]
These definitions take a more familiar form in terms of
the representing map $T$ in Lemma \ref{aux}
One sees immediately for a sequence $\varphi_n \in \sD$ that
\[
\varphi_n \to_\sD \varphi \quad \Leftrightarrow \quad
\varphi_n \to \varphi \quad \mbox{in} \quad \sH \quad \mbox{and} \quad \|T(\varphi_n-\varphi_m)\| \to 0.
\]
Therefore, $\sD$ is \textit{closable} if and only if $T$ is closable,
and, likewise,  $\sD$ is closed if and only if $T$ is closed.

\medskip

An example of a representing map appears in the following construction that
was used in \cite{HS2023a}.
Let $A \in \bB(\sK)$ be a nonnegative contraction in a Hilbert space $\sK$.
The range space $\sA=\ran A^\half$, as a subspace of $\sK$,
is provided with the  semi-inner product
\begin{equation}\label{ipe}
 (A^\half h, A^\half k)_\sA=(\pi h, \pi k)_\sK, \quad h,k \in \sK,
\end{equation}
where $\pi$ is the orthogonal projection in $\sK$ onto
$\cran A^\half=(\ker A^\half)^\perp$. Then
 it is clear that the operator $T \in \bL(\sK, \sH)$
defined by
\[
 A^\half h   \mapsto \pi h, \quad h \in \sH,
\]
with $\dom T=\sA$, is actually a representing map as follows from \eqref{ipe}.

\section{Nondecreasing sequences of linear operators}\label{nondecop}

It will be shown that a sequence of linear operators,
that is nondecreasing in the sense of contractive domination,
as in Definition \ref{domi}, has a linear operator as limit.
The limit will be constructed by means of representing maps.
Moreover, it will be shown that closability and closedness of the
operators are preserved in the limit.

\begin{theorem}\label{Grijp1}
Let $T_n \in \bL(\sH, \sK_n)$, where $\sK_n$ are Hilbert spaces,
be a sequence of linear operators which satisfy
\begin{equation}\label{Grijpa}
T_m \prec_c T_n, \quad m \leq n.
\end{equation}
Then there exists a linear operator $T \in \bL(\sH, \sK)$,
where $\sK$ is a Hilbert space,  such that
\begin{equation}\label{Grijpb}
  \dom T= \left\{ \varphi \in
  \bigcap_{n \in \dN} \dom T_n :\, \sup_{n \in \dN} \|T_n \varphi\| < \infty \right\}
\end{equation}
and which satisfies
 \begin{equation}\label{Grijpc}
 T_n \prec_c T \quad \mbox{and} \quad
  \|T_n \varphi\| \nearrow \|T \varphi\|, \quad \varphi \in \dom T.
 \end{equation}
Moreover, the following statements hold:
\begin{enumerate} \def\labelenumi{\rm(\alph{enumi})}
\item if $T_n$ is closable for all $n \in \dN$, then $T$ is closable;
\item if $T_n$ is closed for all $n \in \dN$, then $T$ is closed.
\end{enumerate}
\end{theorem}

\begin{proof}
Let $T_n$ be a sequence of operators that satisfies \eqref{Grijpa}.
Then it is seen by Cauchy's inequality that the right-hand side $\sD$
in \eqref{Grijpb} is a linear space.
Next the existence of the operator $T$ will be shown.
For each $\varphi \in \sD$ define
\[
 \|\varphi\|_+= \sup_{n \in \dN} \|T_n \varphi\|.
\]
Then $\|\cdot \|_+$ is clearly a well defined seminorm on $\sD$
and let $(\cdot, \cdot)_+$ be the corresponding semi-inner product.
By Lemma \ref{aux} there exists a linear operator $T$ defined on
$\dom T=\sD \subset \sH$ to a Hilbert space $\sK$ such that
\[
 (\varphi, \psi)_+=(T\varphi, T \psi), \quad \varphi \in \sD.
\]
This shows the assertion in \eqref{Grijpc}.

(a) Assume that $T_n$, $n \in \dN$, is closable. To show that $T$ is closable,
it suffices to show that $T=T_{\rm reg}$.
By \eqref{Grijpc} one has
\[
T_n \prec_c T.
\]
Hence  there exist contractions $C_n\in\bB(\sK,\sK_{n})$,
such that $C_n T \subset T_n$ for all $n\in\dN$.
This implies that
\[
   C_n T^{**}     \subset T_n^{**};
\]
see \eqref{dom1}.
In particular, if $\{0,\varphi\} \in T^{**}$, then $\{0,C_{n} \varphi\} \in T_{n}^{**}$,
so that $C_{n}\varphi=0$.
Thus one concludes that $\mul T^{**}\subset \ker C_n$.
Let $P$ be the orthogonal projection from $\sK$ onto $\mul
T^{**}$, then $C_nP=0$.  By means of the Lebesgue decomposition
$T=(I-P)T+PT$, this leads to
\[
 C_n T_{\rm reg}=C_n(I-P)T=C_n[(I-P)T+PT]=C_n T \subset T_n.
\]
Hence, $C_n T_{\rm reg} \subset T_n$ for all $n\in\dN$ and thus
\begin{equation}\label{ass2}
 \|T_n \varphi \|=\|C_n T_{\rm reg} \varphi \|
 \leq \|T_{\rm{reg}} \varphi \| \leq \|T \varphi \|, \quad \varphi \in \dom T.
\end{equation}
Taking the supremum over $n \in \dN$ in \eqref{ass2}
and combining with \eqref{Grijpc} gives
\[
 \|T \varphi \|=\|T_{\rm{reg}} \varphi \|, \quad \varphi \in\dom T.
\]
This implies that $T_{\rm{sing}}=0$ and hence $T$ is closable.

(b) Assume that $T_n$, $n \in \dN$, is closed.
To show that $T$ is closed, let $\varphi_n$
be a sequence in $\dom T$ such that
\begin{equation}\label{gri}
\varphi_n \to \varphi \, \mbox{ in } \,\sH
\quad \mbox{and} \quad
T(\varphi_n-\varphi_m) \to 0 \,\mbox{  in } \, \sK.
\end{equation}
Due to \eqref{Grijpc} one sees that $T_k(\varphi_n-\varphi_m) \to 0$.
Since for each $k \in \dN$ the operator $T_k$ is closed
one obtains that $\varphi \in \dom T_k$
and $T_k(\varphi_n-\varphi) \to 0$ as $n \to \infty$.
In particular, $\varphi \in \bigcap_{n \in \dN} \dom T_n$.
In order to verify that $\varphi \in \dom T$,
first observe that the inequality
\[
 \left| \|T\varphi_n\|-\|T \varphi_m\| \right| \leq \|T(\varphi_n-\varphi_m)\|,
\]
implies, via \eqref{gri}, that $\sup_{m \in \dN} \|T \varphi_m\| <\infty$.
Now it follows from $T_n \varphi_m \to T_n \varphi$ and \eqref{Grijpc} that
\[
  \|T_n \varphi \|
   = \lim_{m \to \infty} \|T_n \varphi_m\|
  \leq \lim_{m \to \infty} \|T \varphi_m\|
  \leq \sup_{m \in \dN}   \|T \varphi_m\| < \infty.
\]
Since this holds for all $n\in \dN$, one concludes that
$\sup_{n \in \dN} \|T_n \varphi \| < \infty$.
Therefore, $\varphi \in \dom T$. Since by (a) the operator $T$ is closable,
it now follows from \eqref{gri}  that $T$ is closed.
\end{proof}

The existence of the limit in Theorem \ref{GrijpA} has been established; however
it is clear that there is no uniqueness. In fact,  this question has been already
addressed in Lemma \ref{aux}. The corollary below is easily verified directly.

\begin{corollary}
Assume the conditions from Theorem {\rm \ref{Grijp1}}
and let $T \in \bL(\sH, \sK)$ be the limit. If $T' \in \bL(\sH, \sK')$,
where $\sK'$ is a Hilbert space,
is another limit with $\dom T'=\dom T$,
then there exists a partial isometry $V \in \bB(\sK, \sK')$
with initial space $\cran T$ and final space $\cran T'$,
such that $T'=VT$.
\end{corollary}

The following simple result is that an operator that
dominates the sequence also dominates the limit.
This fact will have important consequences.

\begin{corollary}\label{Grijp2}
Assume the conditions from Theorem {\rm \ref{Grijp1}}
and let $T' \in \bL(\sH, \sK')$, where $\sK'$ is a Hilbert space,
be a linear operator.  Then
\[
 T_n \prec_c T', \quad n \in \dN \quad \Rightarrow \quad T \prec_c T'.
\]
\end{corollary}

\begin{proof}
 The inequality $T_n \prec_c T'$ implies that  $\dom T' \subset \dom T_n$ and
$\|T_n \varphi \| \leq \| T' \varphi\|$ for $\varphi \in \dom T'$.
Since this holds for all $n \in \dN$, one sees that
\[
 \dom T' \subset \dom T \quad \mbox{and} \quad
 \|T \varphi\|=\sup_{n \in \dN} \|T_n \varphi\| \leq \|T' \varphi\|, \quad \varphi \in \dom T',
\]
in other words $T \prec T'$.
\end{proof}

\section{Nondecreasing sequences of closable operators}\label{nondecclos}

It is a consequence of  Theorem \ref{Grijp1} that
a sequence of closable linear operators  which satisfy \eqref{Grijpa}
 has a closable limit.  The description of the limit
of the closures is of interest.

\begin{proposition}\label{Grijp5}
Let $T_n \in \bL(\sH, \sK_n)$, where $\sK_n$ are Hilbert spaces,
be a sequence of linear operators
for which \eqref{Grijpa} holds
and assume that $T_n$, $n \in \dN$,  is closable.
Let $T$ be the closable limit of $T_n$  in \eqref{Grijpb} and \eqref{Grijpc}.
Then the  closures $T_n^{**} \in \bL(\sH, \sK_n)$ of $T_n$ satisfy
 \begin{equation}\label{GrijpA}
 T_m^{**} \prec_{c} T_n^{**},
 \quad m \leq n, \quad \mbox{and} \quad T_n^{**} \prec_c T^{**}.
\end{equation}
Consequently, there exists a closed linear operator
$S \in \bL(\sH, \sK_{\rm c})$, where $\sK_{\rm c}$ is a Hilbert space,
such that
\begin{equation}\label{GrijpBS}
    \dom S=\left\{ \varphi \in \bigcap_{n \in \dN} \dom T_n^{**} :\,
  \sup_{n \in \dN} \| T_n^{**} \varphi\| < \infty \right\}
\end{equation}
 and which satisfies
\begin{equation}\label{GrijpCS}
 T_n^{**} \prec_c S \prec_c T^{**}
 \quad \mbox{and} \quad
  \|T_n^{**} \varphi\| \nearrow \|S \varphi \|, \quad \varphi \in \dom S.
\end{equation}
In fact, $\dom T^{**} \subset \dom S$, while
$\|S \varphi\|=\|T^{**}\varphi\|$ for all $\varphi \in \dom T^{**}$.
\end{proposition}

\begin{proof}
The sequence $T_n$ is assumed to satisfy \eqref{Grijpa},
thus it follows that
$T_m^{**} \prec_c T_n^{**}$ for  $m \leq n$, by \eqref{dom1}.
Moreover, by Theorem \ref{Grijp1} one has $T_n \prec_c T$,
so that also $T_n^{**} \prec_c T^{**}$ by \eqref{dom1}.
Hence \eqref{GrijpA} holds and, in particular, Theorem \ref{Grijp1} may be applied
to the sequence of closed operators $T_n^{**}$.

Recall from Theorem \ref{Grijp1} that the right-hand side
in \eqref{GrijpBS} is a linear space.
Moreover, by the same theorem there exists a closed linear operator $S$
defined on $\dom S$ in \eqref{GrijpBS} for which \eqref{GrijpCS} holds;
observe that $S \prec_c T^{**}$ by Corollary \ref{Grijp2}.

Now it follows from \eqref{Grijpc} and \eqref{GrijpCS}
that $\|T\varphi\|=\|S \varphi\|$ for all $\varphi \in \dom T$.
Here the operator $S$ is closed and $T$ is closable, and
$S \prec_c T^{**}$ means that $CT^{**}\subset S$ for some
contraction $C \in \mathbf{B}(\sK,\sK_{\rm c})$.
One concludes that $\|S \varphi\|=\|CT^{**}\varphi\|=\|T^{**}\varphi\|$
holds in fact for all $\varphi \in \dom T^{**}$.
\end{proof}

A special case of Theorem \ref{Grijp1}, where all $T_n$ are bounded everywhere
defined operators, is worth mentioning separately.

\begin{corollary}\label{new1}
Let $T_n\in \bB(\sH, \sK_{n})$, where $\sK_n$ are Hilbert spaces,
such that
\[
 \|T_m \varphi \| \le \|T_n \varphi \|, \quad  \varphi \in \sH, \quad m \leq n.
\]
Then there exists a closed linear operator $T \in \bL(\sH, \sK_{\rm c})$,
where $\sK_{\rm c}$ a Hilbert space,  such that
\begin{equation}\label{bbd1}
  \dom T= \left\{ \varphi \in \sH:\,
  \sup_{n \in \dN} \|T_n \varphi\| < \infty \right\}
\end{equation}
and which satisfies
\begin{equation}\label{bbd2}
  T_n \prec_c T \quad \mbox{and} \quad
  \|T_{n} \varphi \| \nearrow \|T \varphi\|, \quad \varphi \in \dom T.
\end{equation}
\end{corollary}

\begin{proof}
This is just an application of Theorem \ref{Grijp1},
as $\bigcap_{n=1}^\infty \dom T_n=\sH$.
Hence there exists a linear operator $T \in \bL(\sH, \sK)$
for which \eqref{bbd1} and \eqref{bbd2} hold.
Since $T_n\in \bB(\sH, \sK_{n})$ one observes that $T_n$,
$n \in \dN$, is closed, which implies that $T$ is closed.
\end{proof}

\begin{remark}
 If in Corollary \ref{new1} one has $\sup_{n \in \dN} \|T_n\| < \infty$,
then $\dom T=\sH$ and $T \in \bB(\sH,\sK)$
by the closed graph theorem.
However, if  $\sup_{n \in \dN} \|T_n\|=\infty$,
then by the uniform boundedness principle
there is an element $\varphi \in \sH$ for which
$\sup_{n \in \dN} \|T_n \varphi\|=\infty$
 and $\dom T$ is a proper subset of $\sH$.
 Note that $\dom T$ is closed
if and only if $T$ is a bounded operator.
 \end{remark}

\section{Nondecreasing sequences of linear relations}\label{nondecrel}

In this section the emphasis will be on nondecreasing sequences of
linear relations in the general case, i.e., the relations are not necessarily
operators or not necessarily closed.
However, also the regular parts and the closures form
nondecreasing sequences. In particular, one may apply
Theorem \ref{nieuw0}, which
leads to a connection with the monotonicity principle in Theorem \ref{mopr}.

\medskip

Let $T_n \in \bL(\sH, \sK_n)$, where $\sK_n$ are Hilbert spaces,
be a sequence of linear relations which satisfy
\begin{equation}\label{vasa0}
T_m \prec_c T_n  \quad m \leq n.
\end{equation}
Observe that the regular parts $T_{n, {\rm reg}} \in \bL(\sH, \sK_n)$
of the relations $T_n$ are closable operators which satisfy
\begin{equation}\label{GrijpAA}
T_{m, {\rm reg}} \prec_c T_{n, {\rm reg}},
 \quad m \leq n,
\end{equation}
see Lemma \ref{regg}.
Hence, by Theorem \ref{Grijp1}, there exists a closable linear operator
$T_{\rm r} \in \bL(\sH, \sK_{\rm r})$, where $\sK_{\rm r}$ is a Hilbert space,
such that
\begin{equation}\label{GrijpBB}
  \dom T_{\rm r}
  =\left\{ \varphi \in \bigcap_{n \in \dN} \dom T_n :\,
  \sup_{n \in \dN} \|T_{n, {\rm reg}} \varphi\| < \infty \right\}
\end{equation}
and which satisfies
\begin{equation}\label{GrijpCC}
 T_{n, {\rm reg}} \prec_c  T_{\rm r}
 \quad \mbox{and} \quad
 \|T_{n, {\rm reg}} \varphi\| \nearrow \| T_{\rm r} \varphi \|,
\quad \varphi \in \dom T_{\rm r}.
\end{equation}
Moreover, the closures $(T_{n, {\rm reg}})^{**} \in \bL(\sH, \sK_n)$
are closed linear operators which satisfy
\begin{equation}\label{GrijpAAc}
(T_{m, {\rm reg}})^{**} \prec_c (T_{n, {\rm reg}})^{**},
 \quad m \leq n, \quad \mbox{and} \quad (T_{n, {\rm reg}})^{**} \prec_c (T_{\rm r})^{**},
\end{equation}
see Proposition \ref{Grijp5}. By the same proposition,
there exists a closed linear operator
$S_{\rm r} \in \bL(\sH, \sK_{\rm c})$, where $\sK_{\rm c}$ is a Hilbert space,
 such that
 \begin{equation}\label{GrijpBB1}
   \dom S_{\rm r}
  =\left\{ \varphi \in \bigcap_{n \in \dN} \dom T_n^{**} :\,
  \sup_{n \in \dN} \|(T_{n, {\rm reg}})^{**} \varphi\|
  < \infty \right\}
\end{equation}
and which satisfies
\begin{equation}\label{GrijpCC1}
 (T_{n, {\rm reg}})^{**} \prec_c  S_{\rm r} \prec_c(T_{\rm r})^{**}
 \quad \mbox{and} \quad
 \|(T_{n, {\rm reg}})^{**} \varphi\| \nearrow \| S_{\rm r} \varphi \|,
\quad \varphi \in \dom S_{\rm r}.
\end{equation}
In fact, $\dom (T_{\rm r})^{**} \subset \dom S_{\rm r}$, while
$\|S_{\rm r} \varphi\|=\|(T_{\rm r})^{**}\varphi\|$ for $\varphi \in \dom (T_{\rm r})^{**}$.

\medskip

It follows from \eqref{vasa0} that the closures $T_n^{**} \in \bL(\sH, \sK_n)$
of $T_n$ are closed relations  which satisfy
\begin{equation}\label{vasa}
T_m^{**} \prec_c T_n^{**}  \quad m \leq n,
\end{equation}
see \eqref{dom1}.
Of course, by Lemma \ref{regg} also the regular parts of $T_n$ satisfy such an inequality;
but this gives again \eqref{GrijpAAc}, due to the identity
\[
 ((T_n)^{**})_{\rm reg}=(T_{n, {\rm reg}})^{**},
\]
see \eqref{vier}.
Since the relation $T_n^{**} \in \bL(\sH, \sK_n)$ is closed, it follows that the product
\[
H_n=T_n^*T_n^{**} \in \bL(\sH)
\]
is a nonnegative selfadjoint relation and
by Theorem \ref{nieuw0} one sees that \eqref{vasa} implies
\[
 H_m \leq  H_n, \quad m \leq n.
\]
Thus according to Theorem \ref{mopr}
there exists a nonnegative selfadjoint relation
$H_\infty \in \bL(\sH)$  which is the limit of the
relations $H_n$ in the strong resolvent sense or, equivalently,
in the strong graph sense.

\begin{theorem}\label{main}
Let $T_n \in \bL(\sH, \sK_n)$, where $\sK_n$ are Hilbert spaces,
be a sequence of  linear relations which satisfy \eqref{vasa0}.
Let $H_\infty \in \bL(\sH)$ be the nonnegative selfadjoint relation,  which is the limit
of the nondecreasing sequence of nonnegative selfadjoint
relations $T_n^* T_n^{**} \in \bL(\sH)$. Then $H_\infty$ satisfies
 \begin{equation}\label{vash}
\dom (H_\infty)^\half=\left\{ \varphi \in \bigcap_{n \in \dN} \dom T_n^{**} :\,
  \sup_{n \in \dN} \| (T_{n, {\rm reg}})^{**} \varphi\| < \infty \right\}
\end{equation}
and, furthermore,
\begin{equation}\label{vasi}
(T_{n, {\rm reg}})^{**} \varphi\| \nearrow  \| (H_{\infty, {\rm op}})^\half \varphi\|,
\quad \varphi \in \dom (H_\infty)^\half.
\end{equation}
Moreover, the limit $S_{\rm r} \in \bL(\sH, \sK_{\rm c})$
of the sequence $(T_{n, {\rm reg}})^{**} \in \bL(\sH, \sK_n)$
in \eqref{GrijpBB1} and \eqref{GrijpCC1} satisfies
\begin{equation}\label{vasj-}
\| (T_{n, {\rm reg}})^{**} \varphi\| \nearrow \|S_{\rm r} \varphi \|,
\quad \varphi \in \dom S_{\rm r}=\dom (H_\infty)^\half.
\end{equation}
Consequently, there exists a partial isometry $U \in \bL(\sK_{\rm c}, \sH)$ such that
\begin{equation}\label{vasj}
 (H_{\infty, {\rm op}})^\half=U S_{\rm r} \quad \mbox{and} \quad
 H_{\infty, {\rm op}} = (S_{\rm r})^* S_{\rm r}.
\end{equation}
\end{theorem}

\begin{proof}
It is clear that the product $H_n=T_n^{*} T_n^{**} \in \bL(\sH)$
is a nonnegative selfadjoint relation. Furthermore,
the closures $T_n^{**}$ of $T_n$ satisfy the inequalities \eqref{vasa}.
Therefore,  the  nonnegative selfadjoint relations
$H_n=T_n^{*} T_n^{**} \in \bL(\sH)$ form a nondecreasing sequence
thanks to Theorem \ref{nieuw0}.
Thus by Theorem \ref{mopr} there exists a nonnegative selfadjoint relation
$H_\infty$ such that \eqref{vasf} and \eqref{vasg} hold. Remember that
\[
 H_n=T_n^{*} T_n^{**}= (T_{n, {\rm reg}})^*(T_{n, {\rm reg}})^{**},
\]
so that there exists a partial isometry $U_n \in \bL(\sK_n, \sH)$, such that
\[
 (H_{n, {\rm op}})^\half= U_n \,(T_{n, {\rm reg}})^{**}.
\]
In other words, \eqref{vasf} and \eqref{vasg} lead to \eqref{vash} and \eqref{vasi}.
Similarly, a comparison of \eqref{GrijpBB} and \eqref{GrijpCC} with
\eqref{vash} and \eqref{vasi} shows that \eqref{vasj-} holds.
Therefore, there exists a partial isometry $U \in \bL(\sL, \sH)$ such that
$(H_{\infty, {\rm op}})^\half= U S_{\rm r}$, which is the first assertion in \eqref{vasj}.
This identity shows that also the second assertion in \eqref{vasj} holds.
\end{proof}

Assume that the sequence $T_n \in \bL(\sH, \sK_n)$ in Theorem \ref{main}
has an upper bound, i.e.,
there exists a linear relation $T \in \bL(\sH, \sK)$,
where $\sK$ is a Hilbert space, such that
\begin{equation}\label{vasa0+}
T_n \prec_c T,  \quad m \leq n.
\end{equation}
For instance, if the sequence $T_n$ consists of operators then $T$ may be taken
as the limit of $T_n$ by Theorem \ref{Grijpa}. It follows from \eqref{vasa0+}  that
\[
 T_{n, {\rm reg}} \prec_c T_{\rm reg} \quad \mbox{and} \quad
  (T_{n, {\rm reg}})^{**} \prec_c  (T_{\rm reg})^{**}.
\]
With these upper bounds it follows for the closable limit  $T_{\rm r}$ of $T_{n, {\rm reg}}$ that
 \[
 T_{\rm r} \prec_c T_{\rm reg} \quad \mbox{and hence} \quad   (T_{\rm r})^{**} \prec_c (T_{\rm reg})^{**}.
\]
Consequently, for the closed limit $S_{\rm r}$ of $(T_{n, {\rm reg}})^{**}$ one has via \eqref{GrijpAAc}
\[
 S_{\rm r} \prec_c  (T_{\rm r})^{**} \prec_c (T_{\rm reg})^{**}.
\]

\section{An example of a nondecreasing sequence}\label{Ex}

In order to illustrate the various possibilities of convergence a simple example
of a nondecreasing sequence will be presented.
Let $R \in \bL(\sH, \sK)$  be a linear operator and define
the sequence of linear operators
$T_n \in \bL(\sH, \sK)$, $n \in \dN$, by
 \begin{equation}\label{griju1-}
 T_n=\sqrt{n}\, R.
 \end{equation}
Then it is clear from \eqref{griju1-} that
\[
\bigcap_{n=1}^\infty \dom T_n=\dom R \quad \mbox{and}
\quad T_n \prec_c T_{n+1},  \quad n \in \dN,
\]
so that \eqref{Grijpa} is satisfied.
Hence one can apply Theorem \ref{Grijp1} to
determine the limit $T$ of the sequence $T_n$.
It follows from \eqref{Grijpb} and  \eqref{Grijpc} that
\begin{equation}\label{griju2-}
 \dom T  =\ker R
 \quad \mbox{and} \quad T=O_{\,\ker R}.
\end{equation}
In fact, it is  clear that $T$ is closable and singular, simultaneously, and that
\begin{equation}\label{griju3-+}
 T^{**}=O_{\,\overline{\ker}\, R}.
\end{equation}
Moreover, observe that it follows from \eqref{griju2-} and \eqref{griju3-+}  that
\[
T^*T= \ker R \times (\ker R)^\perp
 \quad \mbox{and} \quad
 T^*T^{**}= \cker R \times (\ker R)^\perp.
\]
Note that in the special case where $R \in \bB(\sH, \sK)$
this illustrates \cite[Corollary 5.2.13]{BHS}.
If, in addition, the operator $R \in \bL(\sH, \sK)$ is closable,
 then all $T_n$ in \eqref{griju1-} are closable.
The closures $T_{n}^{**}$ of $T_n$ are given by
\[
 T_{n}^{**} = \sqrt{n} \,R^{**},
\]
and it is clear  that \eqref{GrijpA} is satisfied.
Hence one can apply Proposition \ref{Grijp5} to obtain the closed limit $S$
of the sequence $T_{n}^{**}$.
It follows from \eqref{GrijpBS} and \eqref{GrijpCS} that
\begin{equation}\label{griju2-b}
 \dom S   =\ker R^{**} \quad  \mbox{and} \quad S=O_{\,\ker R^{**}}.
\end{equation}
One sees directly from \eqref{griju3-+} that $T^{**} \subset S$,
which illustrates the situation in Proposition \ref{Grijp5}.
The inclusion $T^{**} \subset S$ is strict precisely
when $\overline{\ker} R \subset \ker  R^{**}$ is strict.
As an example where the inclusion is strict, let $R$ be an operator such that
$R^{-1}$  is an operator that is not closable, in which case
$\ker R=\{0\}$ and $\ker R^{**} \neq \{0\}$.
Note that   the nonnegative selfadjoint relation $S^*S$ is given by
\[
 S^*S= \ker R^{**} \times (\ker R^{**})^\perp.
\]
as follows from \eqref{griju2-b}.

Next consider the Lebesgue decomposition of $R$ which is given by
\[
R=R_{\rm reg}+R_{\rm sing}, \quad R_{\rm reg}=(I-P)R, \quad R_{\rm sing}=PR,
\]
where $P$ be the orthogonal projection form $\sK$ onto $\mul R^{**}$.
Then the regular parts $T_{n, {\rm reg}}$ of $T_n$ in \eqref{griju1-} are given by
\[
 T_{n, {\rm reg}}=\sqrt{n}\, R_{\rm reg},
\]
and it is clear  that \eqref{GrijpAA} is satisfied.
For the closable  limit $T_{\rm r}$ of the sequence $T_{n, {\rm reg}}$
it follows from \eqref{GrijpBB} and \eqref{GrijpCC}  that
\[
\dom T_{\rm r} =\ker R_{\rm reg}
\quad \mbox{and} \quad T_{\rm r}=O_{\,\ker R_{\rm reg}}.
\]
 Since $T_{\rm reg}=O_{\rm \ker R}$
 one sees directly that $T_{\rm r} \prec_c T_{\rm reg}$,
which is the general situation.
The inequality is strict precisely when $\ker R \subset \ker R_{\rm reg}$ is strict.
Observe that
\[
 (T_{\rm r})^* T_{\rm r} = \ker R_{\rm reg} \times (\ker R_{\rm reg})^\perp
 \quad \mbox{and}\quad
 (T_{\rm r})^* (T_{\rm r})^{**} = \cker R_{\rm reg} \times (\ker R_{\rm reg})^\perp.
\]
The closures of $T_{n, {\rm reg}}$ are given by
\[
 (T_{n, \rm reg})^{**}=\sqrt{n} (R_{\rm reg})^{**},
\]
and it is clear that \eqref{GrijpAAc} is satisfied.
For the closed  limit $S_{\rm r}$ of the sequence $(T_{n, {\rm reg}})^{**}$
it follows from \eqref{GrijpBB1} and \eqref{GrijpCC1}  that
 \[
 \dom S_{\rm r}=\ker (R_{\rm reg})^{**} \quad \mbox{and} \quad
 S_{\rm r}=O_{\,\ker (R_{\rm reg})^{**}}.
\]
Therefore, one sees that
\begin{equation}\label{ss}
 (S_{\rm r})^* S_{\rm r}
 = \ker (R_{\rm reg})^{**} \times (\ker (R_{\rm reg})^{**})^\perp.
\end{equation}

Finally consider $T_n$ as in \eqref{griju1-}
with a general operator $R \in \bL(\sH, \sK)$.
Then the product relation  $H_n=T_n^* T_n^{**}$ is given by
\[
 H_n=n R^* R^{**}=n (R_{\rm reg})^{*} (R_{\rm reg})^{**}.
\]
Since $\ker (R_{\rm reg})^*(R_{\rm reg})^{**}=\ker (R_{\rm reg})^{**}$,
it follows from Example \ref{haha} that the limit $H_\infty$
of $H_n$ is given by $H_\infty= \ker (R_{\rm reg})^{**} \times (\ker (R_{\rm reg})^{**})^\perp$,
which agrees with \eqref{ss}.

\section{A description of closed linear operators} \label{closed}

Let $T_n \in \bB(\sH, \sK_n)$ be a sequence of operators
that satisfy \eqref{Grijpa}. According to Corollary \ref{new1}
there is a closed limit $T \in \bL(\sH, \sK)$ which
satisfies \eqref{bbd1} and \eqref{bbd2}.
This section contains some variations on this theme.

\medskip

First it is shown that any nonnegative selfadjoint operator is rougly speaking
the limit of a certain class of nonnegative bounded linear operators.

\begin{lemma}\label{labadj}
Let $A \in \bL(\sH)$ be a nonnegative selfadjoint operator.
Then there exists a sequence of nonnegative selfadjoint
operators $A_n \in \bB(\sH)$ such that
\begin{equation}\label{closeda1}
  (A_m  \varphi, \varphi) \leq (A_n \varphi, \varphi), \quad \varphi \in \sH,
   \quad m \leq n,
\end{equation}
and
\begin{equation}\label{closeda2}
   (A_n \varphi, \varphi) \nearrow \|A^\half \varphi\|^2,
   \quad \varphi  \in \dom A^\half.
 \end{equation}
 \end{lemma}

\begin{proof}
Consider the spectral representation of the nonnegative selfadjoint operator $A$
the Hilbert space $\sH$:
 \[
 A=\int_0^\infty \lambda\,dE_\lambda.
\]
By means of this representation
let the nonnegative selfadjoint operators $A_n\in \bB(\sH)$ be defined by
 \[
 A_n=\int_0^n \lambda\,dE_\lambda, \quad n\in\dN.
\]
Then is is clear that $(A_m \varphi, \varphi) \leq (A_n \varphi, \varphi)$, $m \leq n$,
for all $\varphi \in \sH$.  This gives \eqref{closeda1}.
 By the construction of the sequence $A_n$ one obtains
\[
  (A_n \varphi, \varphi) \nearrow \|A^\half  \varphi\|^2,  \quad
  \varphi \in \dom A^\half,
\]
which gives \eqref{closeda2}.
\end{proof}

As a consequence of Lemma \ref{labadj}
there is some kind converse of Corollary \ref{new1}.

\begin{proposition}\label{Grijp41}
Let $T \in \bL(\sH, \sK)$ be a closed linear operator.
Then there exists a sequence of linear operators
$T_n \in \bB(\cdom T, \sH)$, such that
\begin{equation}\label{closed1}
 \|T_m \varphi\| \leq \|T_n \varphi\|, \quad \varphi \in \cdom T, \quad m \leq n,
\end{equation}
and
\begin{equation}\label{closed1+}
   \|T_{n} \varphi \|  \nearrow \|T \varphi \|, \quad \varphi \in \dom T.
\end{equation}
\end{proposition}

\begin{proof}
The product relation $H=T^*T$ is  nonnegative and selfadjoint in $\sH$
with $\mul H=\mul T^*=(\dom T)^\perp$. Then $H= A \hoplus (\{0\} \times (\dom T)^\perp$,
where $A=H_{\rm reg}$ is a nonnegative selfadjoint operator in $\cdom T$.
Then there exists a sequence of nonnegative selfadjoint
operators $A_n \in \bB(\cdom T)$ such that
\[
  (A_m  \varphi, \varphi) \leq (A_n \varphi, \varphi), \quad \varphi \in \cdom T,
   \quad m \leq n,
\]
and
\[
   (A_n \varphi, \varphi) \to \| A^\half \varphi \|^2, \quad \varphi  \in \dom A^\half=\dom T \subset \cdom T.
\]
Due to $H=T^*T$ and $A=H_{\rm reg}$ one sees  that
\[
\|A^\half \varphi\|=\|T\varphi \|,  \quad \varphi=\dom A^\half=\dom T,
\]
see Lemma \ref{regs+}.
Finally define $T_n=A_n^\half$, so that \eqref{closed1} and \eqref{closed1+} are satisfied.
\end{proof}

The last result in this section is a direct consequence of Proposition \ref{Grijp41};
it describes the closability of an operator in terms of a sequence of bounded linear operators;
see for the original statement  \cite[Theorem 8.8, Theorem 8.9]{HSS2018}.

\begin{corollary}
Let $S \in \bL(\sH, \sK)$ be a linear operator.
Then the following statements are equivalent:
\begin{enumerate}\def\labelenumi{\rm(\roman{enumi})}
\item  $S$ is closable;
\item there exists a sequence of  linear operators $T_n \in \bB(\cdom S, \sK_n)$,
where $\sK_n$ are Hilbert spaces, such that
\begin{equation}\label{closed1B}
 \|T_m \varphi\| \leq \|T_n \varphi\|, \quad \varphi \in \cdom S, \quad m \leq n,
\end{equation}
and
\begin{equation}\label{closed1+B}
   \|T_{n} \varphi \|  \nearrow \|S \varphi \|, \quad \varphi \in \dom S.
\end{equation}
\end{enumerate}
\end{corollary}

\begin{proof}
(i) $\Rightarrow$ (ii)
Let $S \in \bL(\sH, \sK)$ be a closable operator and denote its closure by $T$.
Then $T \in \bL(\sH,\sK)$ is a closed operator which extends $S$,
such that $\cdom T=\cdom S$.
Now apply Proposition \ref{Grijp41}.

(ii) $\Rightarrow$ (i) Let $T_n \in \bB(\cdom S, \sK_n)$  be a sequence
for which \eqref{closed1B} holds. Then by Corollary \ref{new1}
there exists a closed linear operator $T \in \bL(\cdom S, \sK)$,
such that
\[
  \dom T= \left\{ \varphi \in \cdom S:\,
   \sup_{n \in \dN} \|T_n \varphi\| < \infty \right\},
\]
and which satisfies
\[
   \|T_{n} \varphi \|  \nearrow \|T \varphi \|, \quad \varphi \in \cdom S.
\]
Thanks to \eqref{closed1+B} one has $\|S \varphi\|=\|T \varphi\|$ for all $\varphi \in \dom S$.
Since $T$ is closed if follows that $S$ is closable.
\end{proof}

An application of these results can be found in \cite[Theorem 6.4]{HS2022},
where pairs of bounded linear operators are classified in terms of almost domination.

\section{Nonincreasing sequences of linear operators}\label{noninc}

In this section there is a brief review for the situation of nonincreasing sequences
of linear operators in the sense of contractive domination.
First recall the analog of the monotonicity principle in Theorem \ref{mopr}
for nonincreasing sequences; see \cite[Theorem~3.7]{BHSW2010}.

\begin{theorem}\label{incr}
Let $K_n \in \bL(\sH)$ be a  sequence of nonnegative selfadjoint  relations
and assume they satisfy
\[
K_n \leq K_m,  \quad m \leq n.
\]
Then there exists a nonnegative selfadjoint relation $K_\infty \in \bL(\sH)$ with
\begin{equation}\label{vvasd1}
K_\infty \leq K_n,  \quad n \in \dN.
\end{equation}
In fact, $K_n \to K_\infty$ in the strong resolvent sense
or, equivalently, in the strong graph sense.
Moreover, the square root of $K_\infty$ satisfies
\begin{equation}\label{vvasf}
\ran (K_\infty)^\half=\left\{ \varphi \in \bigcap_{n \in \dN} \ran (K_n)^\half  :\,
  \lim_{n \to \infty} \| ((K_{n})^{-\half})_{\rm reg} \varphi\| < \infty \right\}
\end{equation}
and, furthermore,
\begin{equation}\label{vvasg}
 \| ((K_{n})^{-\half})_{\rm reg} \varphi\| \nearrow \| ((K_{\infty})^{-\half})_{\rm reg} \varphi\|,
 \quad \varphi \in \ran (K_\infty)^\half.
\end{equation}
\end{theorem}

\begin{proof}
A short proof is included for completeness.
By antitonicity, the sequence $(K_n)^{-1} \in \bL(\sH)$ is nondecreasing;
cf. \cite[Corollary 5.2.8]{BHS}.
Hence, by Theorem \ref{mopr}, there exists a nonnegative selfadjoint relation,
say, $(K_\infty)^{-1} \in \bL(\sH)$, such that $(K_\infty)^{-1}$
is the limit of the sequence $(K_n)^{-1} \in \bL(\sH)$ in the strong resolvent sense
or, equivalently, in the strong graph sense,
and $(K_n)^{-1} \leq  (K_\infty)^{-1}$. Then, again by antitonicity,  $K_\infty \leq K_n$
and, moreover,  $K _\infty$ is the limit of the sequence $K_n$
in the strong graph sense.
The rest of the statements is a direct translation of similar statements in Theorem \ref{mopr}.
\end{proof}

Note that the multivalued parts of the relations $K_n$ in Theorem \ref{mopr}
form a nonincreasing sequence.
If one of the relations $K_n$ in Theorem \ref{incr} is an operator,
then all of its successors are operators and, ultimately,
the limit $K_\infty$ is an operator.

\begin{example}
Let $A \in \bL(\sH)$ be a nonnegative selfadjoint operator or relation.
Then it is clear that the sequence $K_n=\frac{1}{n}A$ of nonnegative selfadjoint operators or relations
is nonincreasing. Hence there exists a nonnegative selfadjoint relation $K_\infty \in \bL(\sH)$
such that $K_n \to K_\infty$ is the strong graph sense.
By means of Example \ref{haha} one sees immediately that
\[
 K_\infty=\cdom A \times \mul A.
\]
\end{example}

\medskip

The following result is the analog of Theorem \ref{Grijp1} for nonincreasing sequences
of linear operators. Due to the sequence being nonincreasing
there are no further convergence
restrictions for the limit as in Theorem \ref{Grijp1}.

\begin{theorem}\label{vGrijp1}
Let $T_n \in \bL(\sH, \sK_n)$, where $\sK_n$ are Hilbert spaces,
be a sequence of linear operators which satisfy
\begin{equation}\label{vGrijpa}
T_n \prec_c T_m, \quad m \leq n.
\end{equation}
Then there exists a linear operator $T \in \bL(\sH, \sK)$,
where $\sK$ is a Hilbert space,  such that
\begin{equation}\label{vGrijpb}
  \dom T=\bigcup_{n \in \dN} \dom T_n,
\end{equation}
and  which satisfies
 \begin{equation}\label{vGrijpc}
 T \prec_c T_n \quad \mbox{and} \quad
  \|T_n \varphi\| \searrow \|T \varphi\|, \quad \varphi \in \dom T.
 \end{equation}
\end{theorem}

\begin{proof}
Denote the right-hand side of \eqref{vGrijpb} by $\sD$. Now let $\varphi \in \sD$,
then clearly $\varphi \in \dom T_N$ for some $N \in \dN$.
For all $n \geq N$ one has $T_n \prec_c T_N$, which implies that $\varphi \in \dom T_n$
for all $n \geq N$ and $\lim_{n \to \infty} \|T_n \varphi \|$ exists by \eqref{vGrijpa}.
Hence for each $\varphi \in \sD$ one may define
\[
 \|\varphi\|_+= \lim_{n \to \infty}  \|T_n \varphi \|.
\]
Then $\|\cdot \|_+$ generates a well-defined seminorm on the linear subspace $\sD$.
Let $(\cdot, \cdot)_+$ be the corresponding semi-inner product.
By Lemma \ref{aux} there exists a linear operator $T$ defined on
$\dom T=\sD \subset \sH$ to a Hilbert space $\sK$ such that
\[
 (\varphi, \psi)_+=(T\varphi, T \psi), \quad \varphi \in \sD.
\]
This shows the assertion in \eqref{vGrijpc}.
\end{proof}

Now Theorem \ref{vGrijp1} will be applied under the assumption that
the linear operators $T_n \in \bL(\sH, \sK_n)$ are closed.
Then the corresponding relations $K_n=T_n^* T_n \in \bL(\sH)$
are nonnegative and selfadjoint.

\begin{theorem}\label{vgrijp4}
 Let $T_n \in \bL(\sH, \sK_n)$, where $\sK_n$ are Hilbert spaces,
be a sequence of closed  linear operators which satisfy \eqref{vGrijpa}
and let $T \in \bL(\sH, \sK)$, where $\sK$ is a Hilbert space,
be the limit operator satisfying  \eqref{vGrijpb} and \eqref{vGrijpc}.
Let $K_\infty \in \bL(\sH)$ be the nonnegative selfadjoint relation,
which is the limit of the nonincreasing sequence of nonnegative selfadjoint
relations $K_n=T_n^* T_n \in \bL(\sH)$, so that  $K_\infty$ satisfies
\eqref{vvasf} and \eqref{vvasg}.
Then $K_\infty$ and $T$ are connected via
\begin{equation}\label{un}
 K_\infty=T^* T^{**}.
\end{equation}
Consequently, there exists a partial isometry $U \in \bB(\sK,\sH)$ such that
\begin{equation}\label{deux}
 (K_{\infty, {\rm reg}})^\half= U (T_{\rm reg})^{**}
 \quad \mbox{or} \quad
 (T_{\rm reg})^{**}=U^*(K_{\infty, {\rm reg}})^\half.
\end{equation}
Moreover, for the limit $T \in \bL(\sH,\sK)$ one has
\begin{enumerate}\def\labelenumi{\rm(\alph{enumi})}
\item $T$ is closable if and only if $T \subset U^*(K_{\infty, {\rm reg}})^\half$;
\item $T$ is closed if and only if $T=U^*(K_{\infty, {\rm reg}})^\half$;
\item $T$ is singular if and only if $K_\infty$ is singular.
\end{enumerate}
\end{theorem}

\begin{proof}
Let $T \in \bL(\sH, \sK)$ be the limit operator in  \eqref{vGrijpb} and \eqref{vGrijpc}.
Recall from \eqref{vGrijpc} that  $T \prec_c T_n$. This leads to
$T^{**} \prec_c (T_n)^{**}=T_n$, which gives $T^*T^{**} \leq T_n^*T_n=K_n$
by Theorem \ref{nieuw0}.
Since this holds for all $n \in \dN$ one obtains
\begin{equation}\label{een1}
 T^*T^{**} \leq K_\infty.
\end{equation}

Moreover, recall  from \eqref{vvasd1} that $K_\infty \leq K_n$, so that
$(K_\infty)^\half \prec_c T_n$ by Theorem \ref{nieuw0}.
In particular, it follows that $(K_{\infty, {\rm reg}})^\half \prec_c T_n$.
Hence one has
\[
 \dom T_n \subset \dom (K_{\infty, {\rm reg}})^\half
 \quad \mbox{and} \quad
 \|(K_{\infty, {\rm reg}})^\half \varphi \| \leq \|T_n \varphi\|, \quad  \varphi \in \dom T_n.
\]
Clearly, with \eqref{vGrijpb}  this now leads to
\[
 \dom T \subset \dom (K_{\infty, {\rm reg}})^\half
 \quad \mbox{and} \quad
 \|(K_{\infty, {\rm reg}})^\half \varphi \| \leq  \inf_{n \in \dN} \|T_n \varphi\|,
 \quad  \varphi \in \dom T.
\]
Thanks to \eqref{vGrijpc} this reads
\[
 \dom T \subset \dom (K_{\infty, {\rm reg}})^\half
 \quad \mbox{and} \quad
 \|(K_{\infty, {\rm reg}})^\half \varphi \| \leq \|T \varphi\|,
 \quad  \varphi \in \dom T,
\]
or equivalently, $(K_{\infty, {\rm reg}})^\half \prec_c T$.
Since closures and regular parts are preserved under the inequality,
this gives $(K_{\infty, {\rm reg}})^\half \prec_c (T^{**})_{\rm reg}$
or $(K_\infty)^\half \prec_c T^{**}$ by Lemma \ref{regg}.
Therefore, one obtains
\begin{equation}\label{twee1}
K_\infty \leq T^* T^{**}.
\end{equation}

\medskip

Combining the inequalities \eqref{een1} and \eqref{twee1} leads to the inequalities
\[
 T^* T^{**} \leq K_\infty \leq T^* T^{**},
\]
or, equivalently,
\[
 (T^*T^{**}-\lambda)^{-1} \leq (K_\infty-\lambda)^{-1} \leq (T^*T^{**}-\lambda)^{-1}, \quad \lambda<0.
\]
This shows  that \eqref{un} holds.
Next \eqref{deux} follows thanks to Lemma \ref{regs+}.

\medskip

Finally, the last assertions concerning the relationship between $T$ and $K_\infty$
will be discussed.

(a) If $T \subset U^*(K_{\infty, {\rm reg}})^\half$, then $T$ is closable.
Conversely, if $T$ is closable, then
$T=T_{\rm reg} \subset (T_{\rm reg})^{**}=U^* (K_{\infty, {\rm reg}})^\half$.

(b) If $T = U^*(K_{\infty, {\rm reg}})^\half$, then $T$ is closed.
Conversely, if $T$ is closed, then
$T=(T_{\rm reg})^{**} =U^* (K_{\infty, {\rm reg}})^\half$.

(c) If $T$ is singular, then $T^*=\sA \times \sB$
where $\sA$ and $\sB$ are closed linear subspaces of $\sK$ and $\sH$, respectively.
Hence $T^{**}=\sB^\perp \times \sA^\perp$, so that  $T^*T^{**}=\sB^\perp \times \sB$
and $K_\infty$ is singular.
Conversely, let $K_\infty=T^*T^{**}$ be singular. Then $T^*T^{**}=\sB^\perp \times \sB$
with a closed linear subspace $\sB$ in $\sH$. Hence it follows that
\[
\left\{
\begin{array}{l}
 \mul T^*=\mul T^*T^{**}=\sB, \\
 \ker T^{**}=\ker T^*T^{**}=\sB^\perp. 
\end{array}
\right.
\]
Therefore $\cran T^{*}= (\ker T^{**})^\perp=\mul T^*$, i.e. $T^*$ and, hence, also $T$ is singular.
\end{proof}

In the present circumstances there is in general no preservation of closedness
in Theorem \ref{vGrijp1}.
This will be shown in the following example; it  is a simple adaptation of
\cite[Example 4.5]{BHSW2010} or \cite[p. 374]{RS1}.

\begin{example}
Let $T_n \in \bL(\sH, \sH \oplus \dC)$ with $\sH=L^2(0,1)$
be given as a column operator by $T_n=\col(T_n^1, T_n^2)$
(see \cite{HS2023a}) with
the operators $T_n^1$ and $T_n^2$ given by
\[
 T_n^1 f=\frac{1}{\sqrt{n}}\, i Df, \quad f(1)=0, \quad \mbox{and} \quad T_n^2 f=f(0).
\]
Here $D$ stands for the maximal differentiation operator in $L^2(0,1)$.
Then $T_n^1$ is closed, $T_n^2$ is singular, while the column $T_n$ is closed.
It is clear that $T_n \prec_c T_m$, $m \leq n$, and the limit $T \in \bL(\sH)$ is given by $Tf=f(0) e$,
where the function $e \in \sH=L^2(0,1)$ is defined by $e(x)=1$.
Moreover, the corresponding nonnegative selfadjoint relation $K_n=T_n^*T_n$ is the operator
in $\sH=L^2(0,1)$ given by
\[
 K_n f =- \frac{1}{n} D^2 f, \quad f'(0)=n f(0), \,\, f(1)=0.
\]
The relations $K_n$ form a sequence that is nonincreasing with the nonnegative selfadjoint limit $K_\infty$
and, by Theorem \ref{vgrijp4}, one has
\[
 K_\infty=T^*T^{**}.
\]
Now observe that $T^*= (\span \{e\})^\perp \times \{0\}$ and $T^{**}=\sH \times \span \{e \}$,
so that $T$ is a singular operator and, in fact $T^*T^{**}=\sH \times \{0\}$.
Hence it follows that
\[
 K_\infty=T^*T^{**}=\sH \times \{0\}.
\]
\end{example}

\section{Appendix: On the products $T^*T$ and $T^*T^{**}$}\label{app}

This appendix contains a number of properties of the
relations $T^*T$ and $T^*T^{**}$ when $T \in \bL(\sH, \sK)$
is a linear relation.  The main emphasis is on the interplay with the regular parts
of these relations. For the convenience of the reader, the arguments are included.

\medskip

Let $T \in \bL(\sH,\sK)$, so that $T^* \in \bL(\sK,\sH)$ is a closed linear relation.
The product $T^*T \in \bL(\sH)$ is defined as
\begin{equation}\label{teetee}
 T^*T=\big\{ \{f,f'\} \in \sH \times \sH :\, \{f,h\} \in T, \, \{h, f'\} \in T^* \,\, \mbox{for some} \,\, h \in \sK \big\}.
\end{equation}
Hence, for the elements in the right-hand side of \eqref{teetee} it is clear that
\begin{equation}\label{teetee1}
 (f',f)=\| h\|^2.
\end{equation}
It follows immediately from \eqref{teetee} and \eqref{teetee1} that the relation $T^*T$
is nonnegative. Moreover, it also follows from \eqref{teetee} and \eqref{teetee1} that
\begin{equation}\label{mul}
 \mul T^*T=\mul T^*.
\end{equation}
It is clear from $T \subset T^{**}$ that  the nonnegative relation
$T^*T$ has a nonnegative  extension
$T^* T^{**}$.  Since $T^{**}$ is closed the product $T^*T^{**}$
is selfadjoint; cf. \cite[Lemma 1.5.8]{BHS}).
Moreover one sees that
\begin{equation}\label{teetee2}
T^*T   \subset T^*T^{**} \subset (T^*T)^*.
\end{equation}
In particular, it follows from \eqref{teetee2} that the closure of $T^*T$ satisfies
\begin{equation}\label{teetee3}
(T^*T)^{**}\subset T^*T^{**}.
\end{equation}
However, in general, even when $T$ is closable,
there is no equality in \eqref{teetee3}.

\medskip

Recall the  definition of the regular part $T_{\rm reg}$:
$T_{\rm reg}=(I-P)T$ where $P$ is the orthogonal projection from $\sK$ onto $\mul T^{**}$,
so that also  $(T^{**})_{\rm reg}=(I-P)T^{**}$. This gives
$(T_{\rm reg})^*=((T^{**})_{\rm reg})^*$, which by taking adjoints  leads to the formal identity
$(T_{\rm reg})^{**}=((T^{**})_{\rm reg})^{**}$.  Note that $(T^{**})_{\rm reg}$
is closed, so that $(T^{**})_{\rm reg}=(T_{\rm reg})^{**}$ in \eqref{vier} is clear.

\medskip

There is an interesting interplay between linear relations and their regular parts
when forming quadratic combinations. Let  $\{f,f'\} \in T^{**}$ and $\{g,g'\} \in T^{*}$,
then by definition there is the identity
\begin{equation}\label{vaag1}
(g',f)=(g,f').
\end{equation}
Recall that the orthogonal projection $P$ maps $\sK$ onto $\mul T^{**}= \cdom T^*$,
and let $Q$ be  the orthogonal projection from $\sH$ onto $\mul T^*=\cdom T^{**}$.
Therefore the identity \eqref{vaag1} reads
\begin{equation}\label{vaag1+}
  (g',(I-Q)f)=((I-P)g,f'),
\end{equation}
which can be rewritten in terms of the regular parts
\begin{equation}\label{heen3}
( (T^*)_{\rm reg} g,f)
=(g, (T_{\rm reg})^{**}f), \,\,\,f \in \dom T^{**}, \,\, g \in \dom T^{*},
\end{equation}
where the equality  \eqref{vier} has been used.
Likewise, there is the identity
\begin{equation}\label{heen33}
( (T^*)_{\rm reg} g,f)
=(g, T_{\rm reg} f), \,\,\,f \in \dom T, \,\, g \in \dom T^{*},
\end{equation}
which also follows from \eqref{vaag1} and \eqref{vaag1+}.
The following lemma shows the various interrelationships.

\begin{lemma}\label{regs}
Let $T \in \bL(\sH, \sK)$ be a linear relation.
Then
\begin{equation}\label{vaag2}
\left\{
\begin{array}{l}
 \big\{ \{\varphi, h \} \in T:\, h \in \dom T^* \big\} \subset T_{\rm reg}, \\
  \big\{ \{\varphi, h \} \in T^{**}:\, h \in \dom T^* \big\} \subset (T^{**})_{\rm reg}=(T_{\rm reg})^{**},
\end{array}
\right.
\end{equation}
and
\begin{equation}\label{zwei}
\left\{
\begin{array}{l}
  T^*T \subset T^* T_{\rm reg} = (T_{\rm reg})^*T_{\rm reg},  \\
  T^*T^{**} = T^* (T_{\rm reg})^{**}=(T_{\rm reg})^*(T_{\rm reg})^{**}.
\end{array}
 \right.
 \end{equation}
Moreover, the  multivalued parts  in \eqref{zwei} satisfy
\begin{equation}\label{zweim}
\mul T^*=\mul (T_{\rm reg})^*,
\end{equation}
and, consequently,
\begin{equation}\label{vierii}
\left\{
\begin{array}{l}
(T^*T)_{\rm reg} \subset (T^*)_{\rm reg}T_{\rm reg}=\big( (T_{\rm reg})^*T_{\rm reg} \big)_{\rm reg}, \\
 (T^*T^{**})_{\rm reg}= (T^*)_{\rm reg} (T_{\rm reg})^{**}=\big( (T_{\rm reg})^*(T_{\rm reg})^{**} \big)_{\rm reg}.
\end{array}
\right.
\end{equation}
In particular,
 \begin{equation}\label{dreii+}
 \left\{
\begin{array}{l}
  \big(\big( (T_{\rm reg})^* T_{\rm reg} \big)_{\rm reg} \varphi, \psi  \big)
   =(T_{\rm reg}  \varphi , T_{\rm reg} \psi), \\
  \hspace{4.8cm}  \varphi \in \dom T^*T_{\rm reg}, \,\, \psi \in \dom T,\\
 \big(\big( (T_{\rm reg})^*(T_{\rm reg})^{**} \big)_{\rm reg} \varphi, \psi  \big)
  =((T_{\rm reg})^{**} \varphi , (T_{\rm reg})^{**}\psi), \\
  \hspace{4.8cm}  \varphi \in \dom T^*T^{**}, \,\, \psi \in \dom T^{**}.
\end{array}
\right.
\end{equation}
\end{lemma}

\begin{proof}
Due to $\dom T^* \subset  \cdom T^*= (\mul T^{**})^\perp$ and \eqref{vier}
one sees that \eqref{vaag2} holds.
Hence it is clear that $T^*T \subset T^*T_{\rm reg}$.
With the orthogonal projection $P$ from $\sK$ onto $\mul T^{**}$, one sees that
\[
  T^*(I-P) T =T^*(I-P)^2 T = ((I-P)T)^* (I-P)T,
\]
which completes the proof of the first part of \eqref{zwei}.
Furthermore, replacing $T$ by $T^{**}$ in the first part of \eqref{zwei}
leads with \eqref{vier} to the second part;
the original inclusion is now an identity since $T^*T^{**}$ is selfadjoint.
The identity \eqref{zweim} is a consequence of \eqref{zwei} due to \eqref{mul}.
The consequence in \eqref{vierii} is obtained from \eqref{zweim} together with \eqref{mul}.

It follows from \eqref{heen33} with
$f=\psi$ and $g=T_{\rm reg} \varphi$ that
\begin{equation*}
\begin{split}
( (T^*)_{\rm reg} T_{\rm reg} \varphi ,\psi)
=(T_{\rm reg} \varphi , T_{\rm reg} \psi), \quad
\varphi , \psi \in \dom T, \quad T_{\rm reg}  \varphi
\in \dom T^*.
\end{split}
\end{equation*}
Note that the conditions
$
\varphi \in \dom T$ and $T_{\rm reg} \varphi \in \dom T^*$
are equivalent to the condition $\varphi \in \dom T^*T_{\rm reg}$.
Thus, with \eqref{vierii}, the first assertion in \eqref{dreii+} has been shown.
Likewise, it follows from \eqref{heen3} with
$f=\psi$ and $g=(T_{\rm reg})^{**} \varphi$ that
\begin{equation*}
\begin{split}
&( (T^*)_{\rm reg}(T_{\rm reg})^{**} \varphi ,\psi)
=((T_{\rm reg})^{**} \varphi , (T_{\rm reg})^{**}\psi),\\
& \hspace{6cm}  \varphi , \psi \in \dom T^{**}, \quad (T_{\rm reg})^{**} \varphi
\in \dom T^*.
\end{split}
\end{equation*}
Note that the conditions
$\varphi \in \dom T^{**}$ and $(T_{\rm reg})^{**} \varphi \in \dom T^*$
are equivalent to the condition $\varphi \in \dom T^*T$, thanks to
\eqref{zwei}. Thus, with \eqref{vierii},
the second assertion in \eqref{dreii+} has been shown.
\end{proof}

There is a special, useful, case of Lemma \ref{regs} that deserves attention.
It is about the orthogonal operator part of $H=T^*T$ when $T$ is closed.

\begin{lemma}\label{regs+}
Let $T \in \bL(\sH, \sK)$ be a closed linear relation and let
$H \in \bL(\sH)$ be the nonnegative selfadjoint relation defined by $H=T^*T$.
Then
\begin{equation}\label{ugh}
 (H_{\rm reg} \varphi, \psi)=( T_{\rm reg} \varphi, T_{\rm reg} \psi),
 \quad \varphi \in \dom T^*T, \,\, \psi \in \dom T,
\end{equation}
and there exists a partial isometry $U \in \bB(\sK, \sH)$ such that
\[
 (H_{\rm reg})^\half = U T_{\rm reg}.
\]
\end{lemma}

\begin{proof}
Recall that $H=T^*T \in \bL(\sH)$ is nonnegative and selfadjoint and that
$\mul H=\mul T^*$. It follows from Lemma \ref{regs} that the identity
\eqref{ugh} is satisfied. Therefore
\[
  \| (H_{\rm reg})^\half \varphi\|=\|T_{\rm reg} \varphi \|,
  \quad \varphi \in \dom H_{\rm reg}=\dom H=\dom T^*T=\dom (T_{\rm reg})^*T_{\rm reg}.
\]
It is clear that $\dom H_{\rm reg}=\dom (T_{\rm reg})^*T_{\rm reg}$ is a core for $(H_{\rm reg})^\half$
and for $T_{\rm reg}$; cf. \cite[Lemma 1.5.10]{BHS}. Hence the assertion follows.
\end{proof}

\end{document}